\documentclass[11pt]{article}
\pdfoutput=1

\usepackage[bib]{preamble/packages}
\usepackage{preamble/commands}
\usepackage{preamble/formatting}

\newcommand*{\C}{\mathcal{C}}

\newcommand*{\E}{\mathbf{E}}

\newcommand*{\N}{\mathbf{N}}
\newcommand*{\Z}{\mathbf{Z}}
\newcommand*{\Q}{\mathbf{Q}}
\newcommand*{\F}{\mathbf{F}}

\DeclareMathOperator{\Thick}{Thick}

\DeclareMathOperator{\rk}{rk}

\newcommand*{\lax}{\mathrm{lax}}

\newcommand*{\Sk}{\mathrm{Sk}}
\newcommand*{\Cell}{\mathrm{Cell}}

\DeclareMathOperator{\filmap}{filmap}

\title{Cellularity of Chromatic Synthetic Spectra}
\author{Shaul Barkan and Sven van Nigtevecht}
\date{2nd May 2025}

\begin{document}
\maketitle

\begin{abstract}
    We show that the $\infty$-category of synthetic spectra based on Morava E-theory is generated by the bigraded spheres and identify it with the $\infty$-category of modules over a filtered ring spectrum. The latter we show using a general method for constructing filtered deformations from t-structures on symmetric monoidal stable $\infty$-categories.
\end{abstract}

Pstr\k{a}gowski \cite{pstragowski_synthetic} defined an $\infty$-category $\Syn_R$ of \emph{$R$-synthetic spectra} categorifying the $R$-Adams spectral sequence in spectra, where $R$ is an Adams type ring spectrum. 
This comes with a natural notion of bigraded spheres, but unlike in spectra, it is not clear in general whether bigraded homotopy groups detect equivalences of $R$-synthetic spectra.
If this happens, we say that $\Syn_R$ is \emph{cellular}.
The main result of this paper is to show that this holds when $R$ is Morava E-theory.
In the following, if $\C$ a stable $\infty$-category with a t-structure, we write $\Fil(\C)$ for $\Fun(\Z^\op,\C)$, and we let $\Wh \colon \C \to \Fil(\C)$ denote its Whitehead filtration functor.

\begin{theoremLetter}
    \label{thm:main}
    Let $E$ denote a Morava E-theory at an arbitrary prime and height.
    \begin{numberenum}
        \item The $\infty$-category $\Syn_E$ is cellular.
        \item There is a symmetric monoidal equivalence
        \[
            \Syn_E \simeqto \Mod_{\map(\nu \S,\, \Wh(\tau^{-1}\nu\S))}(\FilSp)
        \]
        such that $\nu X$ is sent to $\Tot(\Wh(E^{\otimes [\bullet]} \otimes X))$ whenever $X \in \Sp$ is $E$-nilpotent complete.
    \end{numberenum}
\end{theoremLetter}
\begin{proof}
    See \cref{thm:main_theorem} and \cref{cor:Ethy_filtered_model} below.
\end{proof}

Pstr\k{a}gowski showed in \cite{pstragowski_synthetic} that $\MU$-synthetic spectra are cellular.
Later, Lawson \cite{lawson_cellular} generalised this, showing that $\Syn_R$ is cellular whenever $R$ is connective.
This does not imply \cref{thm:main} (except at height $0$): indeed, since $E$ is $\F_p$-acyclic for all $p$, the $\infty$-category $\Syn_E$ is not equivalent to $\Syn_R$ for any connective $R$.

As explained by Burklund--Hahn--Senger \cite[Appendix~C]{burklund_hahn_senger_Rmotivic}, Lawson \cite[Corollary~6.1]{lawson_cellular}, and Pstr\k{a}gowski \cite[Sections~3.3 and~3.4]{pstragowski_perfect_even_filtration}, if $\Syn_R$ is cellular, then this leads to a \emph{filtered model} for $R$-synthetic spectra.
This is the second part of \cref{thm:main}.
Although these types of results are thus well-known, we include a proof in order to highlight the role that t-structures play in obtaining such a result.
Namely, the main hurdle in proving such statements is constructing a $\FilSp$-module structure on $\Syn_R$, or equivalently an action of the monoidal poset $\Z$.
Unlike the discrete monoid $\Z^{\rm disc}$, the monoidal poset $\Z$ has no simple universal property.
We include a short discussion of how t-structures give rise to $\FilSp$-module structures; this is implicit in the previously cited works.

\section*{Acknowledgements}

We would like to thank Robert Burklund for insightful conversations, and for catching a mistake in an earlier draft of this paper.
We also thank Lennart Meier and Tomer Schlank for helpful conversations, and Torgeir Aambø and Marius Nielsen for their continued interest in this paper. 
The second author was supported by the NWO grant \texttt{VI.Vidi.193.111}.

\section{Cellularity}

Fix a homotopy-associative ring spectrum $R$ of Adams type.
We follow the terminology and notation of \cite{pstragowski_synthetic}.
In addition, it will be useful to introduce the following concept.

\begin{definition}
	A cofibre sequence $X \to Y \to Z$ of spectra is called \defi{$R$-exact} if it yields a short exact sequence on $R_{\ast}$-homology:
	\[
        0 \to R_{\ast} X \to R_{\ast} Y \to R_{\ast}Z  \to 0.
    \]
\end{definition}

\begin{definition}
    A (not necessarily stable) full subcategory of $\Sp$ is called \defi{$R$-thick} if it is closed under
    \begin{numberenum}
        \item arbitrary suspensions,
        \item retracts,
        \item 2-out-of-3 for $R$-exact cofibre sequences.
    \end{numberenum}
    For a collection of spectra $J \subseteq \Sp$, we denote by $\Thick^{R}(J)$ the smallest $R$-thick subcategory of $\Sp$ containing $J$.
\end{definition}

It is straightforward to see that the subcategory $\Sp_R^\fp$ of finite $R$-projective spectra is $R$-thick.
It follows that
\[
    \Thick^R(\S) \subseteq \Sp_R^\fp.
\]

\begin{proposition}
    \label{prop:characterisation_cellularity}
    If we have an equality
    \[
        \Thick^R(\S)= \Sp_R^{\fp},
    \]
    then $\Syn_R$ is cellular.
\end{proposition}

We do not know if this is a necessary condition for $\Syn_R$ to be cellular.

\begin{proof}
    The $\infty$-category $\Syn_R$ is generated under colimits by $\Thick(\nu P \mid P \in \Sp_R^\fp)$; see \cite[Remark~4.14]{pstragowski_synthetic}.
    It therefore suffices to show that this thick subcategory is contained in $\Thick(\nu \S^n \mid n\in\Z)$.
    The assumption $\Sp_R^\fp \subseteq \Thick^R(\S)$ implies that
    \[
        \nu(\Sp_R^\fp) \subseteq \nu \Thick^R(\S).
    \]
    Recall that $\nu$ is additive and sends $R$-exact cofibre sequences to cofibre sequences; see \cite[Lemma~4.23]{pstragowski_synthetic}.
    Using this and the fact that $\Sp$ is idempotent-complete, it follows from the definition of an $R$-thick subcategory that we have the containment $\nu\Thick^R(\S) \subseteq \Thick(\nu \S^n \mid n\in\Z)$.
    This finishes the proof.
\end{proof}

\begin{theorem}
    \label{thm:main_theorem}
    Let $p$ be a prime and let $h\geq 0$ be an integer, and let $E$ denote a Morava E-theory at height $h$.
    Then $\Syn_E$ is cellular.
\end{theorem}

In the proof, let $K$ denote a Morava K-theory corresponding to $E$.

\begin{lemma}\label{nakayama-E-theory}
    Let $X \to Y$ be a map of finite $E$-projective spectra.
    Then $E_{\ast}X \to E_{\ast}Y $ is surjective if and only if $K_{\ast}X \to K_{\ast}Y$ is surjective.
\end{lemma}
\begin{proof}
	Before we prove this, let us recall that since $E_{\ast}$ is a graded power series ring over a local ring, every graded projective module over it is free.
    It follows that for every $P \in \Sp^{\fp}_E$, we have a natural isomorphism 
	\[
        K_{\ast} \otimes_{E_{\ast}} E_{\ast}P = \pi_{\ast}(K)\otimes_{\pi_{\ast}E} \pi_{\ast}(E \otimes P) \congto \pi_{\ast} (K \otimes_E (E \otimes P)) = \pi_{\ast}(K \otimes P) = K_{\ast} P.
    \]
    We proceed to the proof.
	
	($\implies$) Suppose that $E_{\ast}X \to E_{\ast}Y$ is surjective.
    Tensoring this with $K_*$ yields a surjective map, which by the above isomorphism can be identified with the map $K_*X \to K_*Y$.
	
	($\impliedby$) Suppose that $K_{\ast}X \to K_{\ast}Y$ is surjective, and write $C_{\ast} = \coker(E_{\ast}X \to E_{\ast}Y)$.
    The isomorphism above and right exactness of the tensor product imply that $\coker(K_{\ast}X \to K_{\ast}Y)  \cong K_{\ast} \otimes_{E_{\ast}} C_{\ast} = 0$.
    Since $C_{\ast}$ is finitely generated, Nakayama's lemma applies, which implies that $C_{\ast} = 0$, so that $E_{\ast}X \to E_{\ast}Y$ is surjective.
\end{proof}

\begin{proof}[Proof of \cref{thm:main_theorem}]
    By \cref{prop:characterisation_cellularity}, it suffices to show that every finite $E$-projective spectrum $P$ is in $\Thick^E(\S)$.
    By suspending sufficiently many times, we may without loss of generality assume that $P$ is connective.
    We will now proceed by ascending induction on $\N \times \N$ with the lexicographical ordering, where to every finite $E$-projective spectrum $P$ we assign
    \[
        (\dim P,\ \rk P) \in \N \times \N,
    \]
    where $\dim P$ denotes the dimension of the top cells of $P$, and where $\rk P \defeq \rk_{E_{\ast}} (E_{\ast}P)$.
    For the base case, where $\dim P = 0$, it follows from connectivity of $P$ that $P$ is a sum of zero-spheres, which is obviously in $\Thick^E(\S)$.
    We proceed to the inductive step.
    
    As a first case, assume that there exists a top-dimensional cell (of dimension $d \defeq \dim P$) for which the projection onto the top cell $P \to \S^d$ induces a surjection on $E_*$-homology.
    Let $F$ denote the fibre of this map.
    A surjection of projective modules is always split, so $E_{\ast}F = \ker(E_{\ast} P \to E_{\ast} \S^d )$ is projective, and therefore $F \in \Sp^{\fp}_E$.
    Moreover, we have $\rk F = \rk P - 1$ and $\dim F \le \dim P$, so by our induction hypothesis, we have $F \in \Thick^{E}(\S)$.
    Since $F \to P \to \S^{d}$ is an $E$-exact fibre sequence, it follows that $P \in \Thick^{E}(\S)$.
	
	We may therefore assume that for every choice of top cell for $P$, the projection onto the top cell is not surjective on $E_{\ast}$.
	A choice of cellular filtration on $P$ gives us a cofibre sequence
	\[
        \Sk_{d-1}(P) \to P \to \bigoplus_{\Cell_d(P)} \S^{d},
    \]
	where $\Cell_d(P)$ is the set of $d$-dimensional cells for the chosen cellular filtration.
    Consider now the map on K-theory induced by projection onto a top-dimensional cell
	\[
        K_{\ast}P \to K_{\ast} \S^{d} \cong K_{\ast - d}.
    \]
	As $K_{\ast}$ is a (graded) field and the right hand side is of rank $1$, this map is either surjective or zero.
    If it were surjective, then by \cref{nakayama-E-theory} it would have also been surjective on $E_{\ast}$, in contradiction to our assumption.
    It must therefore be the zero map on $K_{\ast}$.
    As this argument applies to all top cells, we conclude that $K_{\ast}(P) \to K_{\ast}( \bigoplus_{\Cell_d(P)} \S^{d} )$ is also the zero map. It follows that $\Sk_{d-1}(P) \to P$ is surjective on $K_{\ast}$, so by \cref{nakayama-E-theory} it is also surjective on $E_{\ast}$.
    We learn that the cofibre sequence
    \[
        \bigoplus_{\Cell_{d}(P)} \S^{d-1} \to \Sk_{d-1}(P) \to P
    \]
    is $E_{\ast}$-exact.
    But $\dim(\Sk_{d-1}(P)) = \dim(P)-1$, so by the induction hypothesis, we have $\Sk_{d-1}(P) \in \Thick^{E}(\S)$.
    We conclude that $P \in \Thick^{E}(\S)$, and we are done.
\end{proof}

\section{Deformations from t-structures}

The $\infty$-category of \defi{filtered spectra} is defined as $\FilSp = \Fun(\Z^\op,\Sp)$, where we consider $\Z$ as a poset under the usual ordering.
We regard this as a symmetric monoidal $\infty$-category under Day convolution (where $\Z$ carries addition); note that this turns it into a presentably symmetric monoidal $\infty$-category, i.e., an $\E_\infty$-algebra in $\PrL$.
This category comes with a notion of shifting: if $X$ is a filtered spectrum, then we write $X(n)$ for the filtered spectrum given by
\[
    \begin{tikzcd}
        \Z^\op \rar["-n"] & \Z^\op \rar["X"] & \Sp.
    \end{tikzcd}
\]
Note that this functor is equivalently given by tensoring with $\1(n)$, where $\1$ denotes the unit of $\FilSp$.
The connecting maps of $X$ induce natural transformations $X(n+1)\to X(n)$ for every $n$.

Following \cite{barkan_monoidal_algebraicity}, a \defi{deformation} is a module over $\FilSp$ in $\PrL$.
If $\C$ is a deformation and $X\in \C$, then we can define $X(n)$ as $\1(n) \otimes X$.
This gives rise to a filtered mapping spectrum: for $X,Y\in\C$, we define $\filmap_\C(X,Y)$ to be the filtered spectrum given by
\[
    \map_\C(\1(\blank)\otimes X,\, Y),
\]
where $\map_\C$ denotes the mapping spectrum functor of the stable $\infty$-category $\C$.

\begin{theorem}[Filtered Schwede--Shipley; Pstr\k{a}gowski \cite{pstragowski_perfect_even_filtration}, Proposition~3.16]
    \label{thm:filtered_schwede_shipley}
    Let $\C$ be a deformation, and let $X \in \C$.
    Then the following are equivalent.
    \begin{letterenum}
        \item The functor $\filmap_\C(X,\blank)\colon \C \to \Mod_{\filmap_\C(X,X)}(\FilSp)$ is a symmetric monoidal equivalence.
        \item The object $X$ is compact, and the objects $X(n)$ for $n\in\Z$ generate $\C$ under \textbr{de}suspensions and colimits.
    \end{letterenum}
\end{theorem}

Applying this result requires obtaining a $\FilSp$-module structure on $\C$.
As noted in the introduction, obtaining this structure can be difficult, because the poset $\Z$ is not free as a symmetric monoidal $\infty$-category.
An important source of such a structure is a monoidal t-structure on $\C$, as we now explain.

Suppose $\C$ is a symmetric monoidal stable $\infty$-category with a compatible t-structure.
Recall that the Whitehead filtration functor $\Wh \colon \C \to \Fil(\C)$ is lax symmetric monoidal, being the composite
\[
    \begin{tikzcd}
        \C \rar["\mathrm{Const}"] & \Fil(\C) \rar["\tau_{\geq 0}^{\mathrm{diag}}"] & \Fil(\C)
    \end{tikzcd}
\]
where the first functor is the constant-filtered-object functor, and the second is the connective cover with respect to the diagonal t-structure (see, e.g., \cite[Proposition~II.1.23]{hedenlund_phd}), both of which are canonically lax symmetric monoidal.
If $A$ is an $\E_\infty$-algebra in $\C$, we therefore obtain an $\E_\infty$-algebra $\Wh A$ in $\Fil(\C)$, which by the equivalence $\CAlg(\Fil(\C)) \simeq \Fun^\lax(\Z^\op,\C)$ of \cite[Example~2.2.6.9]{HA} is the same as a lax symmetric monoidal structure on the functor $\Wh A \colon \Z^\op \to \C$.

\begin{definition}
    \label{def:t_strict}
    Let $\C$ be a symmetric monoidal $\infty$-category with a compatible t-structure.
    Let $A$ be an $\E_\infty$-algebra in $\C$.
    We say that $A$ is \defi{t-strict} if the lax symmetric monoidal functor $\Wh A \colon \Z^\op \to \C$ is strong symmetric monoidal.
\end{definition}

\begin{lemma}
    Let $\C$ and $A$ be as in \cref{def:t_strict}.
    Then $A$ is t-strict if and only if
    \begin{letterenum}
        \item The map $\1 \to \tau_{\geq 0} A$ induced by the unit of $A$ is an isomorphism.
        \item The natural map $\tau_{\geq n} A \otimes \tau_{\geq m} A \to \tau_{\geq n+m}A$ is an isomorphism for every $n,m\in\Z$.
    \end{letterenum}
\end{lemma}
\begin{proof}
    Condition~(a) says that $\Wh A$ preserves empty products, while condition~(b) says it preserves binary products.
\end{proof}

\begin{theorem}\label{fil-morita-t-struct}
    Let $\C$ be a presentably symmetric monoidal stable $\infty$-category equipped with a compatible t-structure.
    Let $A$ be an $\E_\infty$-algebra in $\C$.
    Suppose that
    \begin{letterenum}
        \item $A$ is t-strict,
        \item the unit $\1$ of $\C$ is compact,
        \item the objects $\opSigma^{n}\tau_{\geq m}A$ for $n,m\in\Z$ generate $\C$ under colimits.
    \end{letterenum}
    Then the functor
    \[
        \map_{\C}(\1,\ \Wh A \otimes \blank)\colon  \C \simeqto \Mod_{\map_{\C}(\1,\Wh A)} (\FilSp)
    \]
    is an equivalence of symmetric monoidal $\infty$-categories.
\end{theorem}
\begin{proof}
    The symmetric monoidal functor
    \[
        \begin{tikzcd}
            \Z \rar["n\,\mapsto\, -n"] &[2em] \Z^\op \rar["\Wh A"] &[1.5em] \C,
        \end{tikzcd}\qquad n \mapsto \optau_{\geq -n}A
    \]
    induces a symmetric monoidal left adjoint $\FilSp \to \C$, giving $\C$ the structure of a symmetric monoidal deformation.
    Applying \cref{thm:filtered_schwede_shipley}, all that remains is to identify $\filmap_\C(\1,\blank)$ with $\map(\1,\,\Wh A\otimes \blank)$.
    By assumption of t-strictness, we see that for every $n\in\Z$, we have a natural (in $n$) isomorphism
    \[
        \map_\C(\1,\, \tau_{\geq n}A) \cong \map_\C(\tau_{\geq -n}A,\, \1).
    \]
    The right-hand side is (by definition of the deformation structure) the value at filtration~$n$ of the filtered spectrum $\filmap_\C(\1,\1)$, proving the claim.
\end{proof}

\begin{corollary}
    \label{cor:Ethy_filtered_model}
    Let $E$ denote a Morava E-theory.
    Then there is a symmetric monoidal equivalence
    \[
        \Syn_E \simeqto \Mod_{\map(\nu \S,\, \Wh(\tau^{-1}\nu\S))}(\FilSp)
    \]
    such that $\nu X$ is sent to $\Tot(\tau_{\ge \star}(E^{\otimes [\bullet]} \otimes X))$ whenever $X$ is a $E$-nilpotent complete spectrum.
\end{corollary}
\begin{proof}[Proof of \cref{cor:Ethy_filtered_model}]
    In the case $\C= \Syn_R$, the $\tau$-inverted unit $\tau^{-1}\nu\S$ is a t-strict $\E_\infty$-algebra.
    The resulting functor $\map(\1,\Wh A)$ is called the \emph{signature functor} in \cite{christian_jack_synthetic_j,CDvN_part1,CDvN_part2}, where it is denoted by $\sigma$; in \cite[Appendix~C]{burklund_hahn_senger_Rmotivic}, this functor is denoted by $i_*$.
    The result now follows by using \cite[Proposition~1.25]{CDvN_part1}.
\end{proof}

\begin{remark}
    For general $R$, the underlying filtered spectrum $\sigma \nu(\S)$ is, after completion, equivalent to the d\'ecalage of the cosimplicial Adams resolution:
    \[
        \Tot(\tau_{\geq \star} (R^{\otimes [\bullet]})).
    \]
    If $R$ is an $\E_\infty$-ring, then this equivalence is naturally one of filtered $\E_\infty$-rings; see \cite[Proposition~1.25]{CDvN_part1}.
    If the sphere is $R$-nilpotent complete, then this is even true without completion of filtered spectra.
    Alternatively, as in \cite[Proposition~C.22]{burklund_hahn_senger_Rmotivic}, if $R$ is an $\E_\infty$-ring, the above implies that $\sigma$ induces a symmetric monoidal equivalence
    \[
        \Mod_{\sigma(\nu \S)_\tau^\wedge}(\Syn_R) \simeqto \Mod_{\Tot(\tau_{\geq \star}(R^{[\bullet]}))}(\FilSp).
    \]
\end{remark}

\printbibliography[heading=bibintoc]

@misc{barkan_monoidal_algebraicity,
  title = {Chromatic {{Homotopy}} Is {{Monoidally Algebraic}} at {{Large Primes}}},
  author = {Barkan, Shaul},
  date = {2023-06-05},
  eprint = {2304.14457v2},
  eprintclass = {math},
  archiveprefix = {arXiv},
  pubstate = {Preprint}
}

@misc{burklund_hahn_senger_Rmotivic,
  title = {Galois reconstruction of Artin-Tate $\mathbb{R}$-motivic spectra},
  author = {Burklund, Robert and Hahn, Jeremy and Senger, Andrew},
  date = {2022-06-30},
  eprint = {2010.10325v2},
  eprintclass = {math},
  archiveprefix = {arXiv},
  pubstate = {Preprint}
}

@misc{CDvN_part1,
  title = {Descent Spectral Sequences through Synthetic Spectra},
  author = {Carrick, Christian and Davies, Jack Morgan and van Nigtevecht, Sven},
  options = {useprefix=true},
  date = {2024-07-01},
  eprint = {2407.01507v1},
  eprintclass = {math},
  archiveprefix = {arXiv},
  pubstate = {Preprint}
}

@misc{CDvN_part2,
  title = {The Descent Spectral Sequence for Topological Modular Forms},
  author = {Carrick, Christian and Davies, Jack Morgan and van Nigtevecht, Sven},
  options = {useprefix=true},
  date = {2024-12-02},
  eprint = {2412.01640v1},
  eprintclass = {math},
  archiveprefix = {arXiv},
  pubstate = {Preprint}
}

@misc{christian_jack_synthetic_j,
  title = {A synthetic approach to detecting $v_1$-periodic families},
  author = {Carrick, Christian and Davies, Jack Morgan},
  date = {2024-01-29},
  eprint = {2401.16508},
  eprintclass = {math},
  archiveprefix = {arXiv},
  pubstate = {Preprint}
}

@unpublished{HA,
  title = {Higher {{Algebra}}},
  author = {Lurie, Jacob},
  date = {2017-09-18},
  url = {https://www.math.ias.edu/~lurie/papers/HA.pdf},
  shorthand = {HA}
}

@phdthesis{hedenlund_phd,
  title = {Multiplicative {{Tate Spectral Sequences}}},
  author = {Hedenlund, Alice Petronella},
  date = {2020},
  publisher = {University of Oslo},
  issn = {1501-7710},
  url = {https://www.mn.uio.no/math/personer/vit/rognes/theses/hedenlund-thesis.pdf},
  school = {University of Oslo}
}

@misc{lawson_cellular,
  title = {Synthetic Spectra Are (Usually) Cellular},
  author = {Lawson, Tyler},
  date = {2024-02-06},
  eprint = {2402.03257},
  eprintclass = {math},
  archiveprefix = {arXiv},
  pubstate = {Preprint}
}

@misc{pstragowski_perfect_even_filtration,
  title = {Perfect Even Modules and the Even Filtration},
  author = {Pstrągowski, Piotr},
  date = {2024-10-24},
  eprint = {2304.04685v2},
  eprintclass = {math},
  archiveprefix = {arXiv},
  pubstate = {Preprint}
}

@article{pstragowski_synthetic,
  title = {Synthetic Spectra and the Cellular Motivic Category},
  author = {Pstrągowski, Piotr},
  date = {2022-12-20},
  journaltitle = {Inventiones mathematicae},
  shortjournal = {Invent. math.},
  issn = {0020-9910, 1432-1297},
  doi = {10.1007/s00222-022-01173-2},
  langid = {english}
}

\end{document}